\documentclass[a4paper]{amsart}

\usepackage{amssymb, latexsym, amsthm, amssymb, amsfonts, amsmath, amsopn, amstext, amscd, latexsym}
\numberwithin{equation}{section}
\usepackage{enumerate, xcolor, url, hyperref}
\usepackage[right]{lineno}

\newcommand{\stab}{G_x}

\newcommand{\abs}[1]{|#1|}
 \def\image{\operatorname{Im}}
\def\ker{\operatorname{Ker}}

\newcommand{\allone}{{\mathbf 1}}
\newcommand{\C}{{\mathbb C}}
\newcommand{\Q}{{\mathbb Q}}

\DeclareMathOperator{\Der}{Der}
\DeclareMathOperator{\Alt}{Alt}
\DeclareMathOperator{\Sym}{Sym}
\DeclareMathOperator{\SL}{SL}
\DeclareMathOperator{\GL}{GL}
\DeclareMathOperator{\ASL}{ASL}
\DeclareMathOperator{\Aut}{Aut}

\DeclareMathOperator{\PSL}{PSL}

\DeclareMathOperator{\agl}{AGL}

\DeclareMathOperator{\Sp}{Sp}

\DeclareMathOperator{\Sz}{Sz}
\DeclareMathOperator{\Ree}{Ree}
\DeclareMathOperator{\PSU}{PSU}
\DeclareMathOperator{\AGL}{AGL}
\DeclareMathOperator{\AGmL}{A\Gamma L}

\DeclareMathOperator{\ind}{ind}

\newtheorem{theorem}{Theorem}[section]
\newtheorem{lemma}[theorem]{Lemma}
\newtheorem{proposition}[theorem]{Proposition}

\newtheorem{question}[theorem]{Question}
\newtheorem{corollary}[theorem]{Corollary}

\newtheoremstyle{plainsl}%
	{4pt} 
	{4pt}
	{\normalfont}  
	{}
	{\normalfont\bfseries}
	{.}
	{ }
	{}

\theoremstyle{plainsl}
\newtheorem{example}[theorem]{Example}
\newtheorem{remark}[theorem]{Remark}
\newtheorem{defn}[theorem]{Definition}

\begin{document}

\title[EKR-module property for $2$-transitive groups]{All $2$-transitive groups have the EKR-module property}

\author[Karen Meagher and  Peter Sin]{Karen Meagher${^*}$ and  Peter Sin$^{\dagger}$}
\address{Department of Mathematics and Statistics, University of Regina,
  Regina, Saskatchewan S4S 0A2, Canada}\email{karen.meagher@uregina.ca}
\address{Department of Mathematics, University of Florida, P. O. Box 118105, Gainesville FL 32611, USA}\email{sin@ufl.edu}

\thanks{${^*}$Research supported in part by an NSERC Discovery Research Grant,
    Application No.: RGPIN-03852-2018.}
\thanks{${^\dagger}$Research partially supported by a grant from the Simons
  Foundation (\#633214 to Peter Sin).}

\begin{abstract}
  We prove that every 2-transitive group has a property called the
  {\it EKR-module property}. This property gives a characterization
  of the maximum intersecting sets of permutations in the
  group. Specifically, the characteristic vector of any maximum
  intersecting set in a 2-transitive group is a linear combination
  of the characteristic vectors of the stabilizers of points and
  their cosets. We also consider when the derangement graph of a
  2-transitive group is connected and when a maximum intersecting set
  is a subgroup or a coset of a subgroup.
\end{abstract}

\subjclass[2010]{Primary 05C35; Secondary 05C69}

\keywords{derangement graph, independent sets, Erd\H{o}s-Ko-Rado
  theorem, 2-transitive groups}

\date{today}

\maketitle

\section{Introduction}

The Erd\H{o}s-Ko-Rado (EKR) Theorem~\cite{MR0140419} is a major result
in extremal set theory.  This famous result gives the size and the
structure of the largest collections of pairwise intersecting $k$-subsets of
an $n$-set.  The Erd\H{o}s-Ko-Rado Theorem has been generalized in
many different ways. One generalization is to show that a version of
the theorem holds for different objects. To date, a version of the EKR
Theorem has been shown to hold for the following objects: $k$-subsets
of an $n$-set~\cite{MR1429238, MR0140419, MR0771733}, integer
sequences~\cite{MR657052}, $k$-dimensional subspaces of an
$n$-dimensional vector space over a finite field~\cite{MR867648}, signed
sets~\cite{MR2320597}, partitions~\cite{MR2423345} and perfect
matchings~\cite{MR3646689}, as well as many other objects.

The commonality relating these results is that a largest set of
(pairwise) intersecting objects must be a set of objects that
intersect in a ``canonical'' way.  For example, a largest set of
intersecting $k$-sets is the collection of all $k$-sets that contain a
common point. A largest set of intersecting $k$-subspaces is the set
of all subspaces that contain a common 1-dimensional
subspace. Similarly, a largest set of intersecting perfect matchings
is the collection of all perfect matchings that contain a fixed
pair. In all of these cases, the objects are sets of elements and two
objects are said to intersect if they contain a common element.  And
for all the cases named above, a largest set of intersecting objects
is the collection of all objects that contain a fixed element---these
are the canonical intersecting sets.

In general, whenever we have objects formed from elements we can ask
``what is the size and structure of a largest set of intersecting
objects?''. If a largest intersecting set must be a canonical
intersecting set, then  we say that a version of the EKR Theorem
holds.

In this paper we consider permutations.
Two permutations $g,h \in \Sym(n)$ {\it intersect} if there exists an
$i \in \{1, \dots, n\}$ with $i^g = i^h$. (Here a permutation $g$
is the object, and the elements that form it are the pairs $(i,j)$
where $i^g=j$.)

Let $G$ be a transitive subgroup of $\Sym(n)$.
Clearly the stabilizer in $G$ of a point, or the coset of a stabilizer of a point, is an
intersecting set of permutations. These sets are denoted by
\[
S_{i, j} = \{g \in G \,|\, i^g = j \},
\]
where $i, j \in \{1,\dots, n\}$ and we call them the {\it canonical
  intersecting sets}.

The intersecting sets of largest size in $G$ are called the {\it maximum intersecting sets}.
We say that a group $G$ has the {\it EKR property} if the canonical intersecting sets are maximum intersecting sets.  The
group $G$ is further said to have the {\it strict-EKR property} if
the canonical intersecting sets are the only maximum intersecting sets.
(Note that these properties depend on the group
action.)  Many specific groups have been shown to have either the
strict-EKR property, or the EKR property~\cite{MR3415003, MR3780424,
  KaPa2, MR3923591, MR3921038}. One of the most general results is the
following, which is equivalent to every 2-transitive group having the EKR
property.

\begin{theorem}[Theorem 1.1~\cite{MR3474795}] \label{thm:2transEKR}
  Let $G$ be a finite $2$-transitive permutation group on the set $\{1,\dots,n\}$. The cardinality of a largest intersecting set in $G$ is $|G|/n$.
\end{theorem}

Clearly any group that has the strict-EKR property has the EKR property. There are 2-transitive permutation groups that do not have the strict-EKR
property, for example $PGL_n(q)$ has the strict-EKR property if and
only if $n=2$~\cite{KaPa, MR3921038}. 

In this paper we consider a property
related to the EKR property, and the strict-EKR property; this
property is called the {\it EKR-module property}. The EKR-module
property was first defined in~\cite{MR3923591}, the definition we give
here is slightly different, but equivalent. 

Before defining the EKR-module property, we need some
notation. The {\it regular module} of $G$ is the (complex) vector space with basis $G$.
We can think of its elements as vectors of length $\abs{G}$.
For any $S \leq G$ define the {\it characteristic vector} of $S$ to
be the vector with entry $1$ in position $g$ if
$g \in S$ and $0$ otherwise; this vector is denoted by $v_S$.
We denote the characteristic vector of $S_{i,j}$ by $v_{i,j}$.

\begin{defn}\label{ekr-module}
A transitive permutation group $G$ has the {\it EKR-module property} if, for any maximum intersecting set of permutations $S$ in $G$,  the characteristic vector
$v_S$ is a linear combination of the vectors $v_{i, j}$ with
$i, j \in \{1,\dots, n\}$. 
\end{defn}

Like the EKR property and strict-EKR property, this is a property of the group action. The main result of
this paper is the following.

\begin{theorem}\label{thm:main}
  Any $2$-transitive group has the EKR-module property.
\end{theorem}

This result was conjectured in~\cite[Conjecture 1.3]{MR3474795}.
We feel that this is the most general statement for all 2-transitive
groups, in the context of EKR-type results. The theorem also gives
information about the structure of the maximum intersecting sets in a
2-transitive group; this is described in detail in
Section~\ref{sec:consequences}.

Part of the motivation for Definition~\ref{ekr-module} comes from  several papers~\cite{GoMe, MR3780424, KaPa2, MR3921038} which prove that a group has the strict-EKR property by first showing that the group has the EKR property, and then showing the group has the EKR-module property. The final characterization is achieved by showing the only linear combinations of the vectors $\{ v_{i,j}\,|\, i,j \in\{1,\dots,n\} \}$ that give a characteristic vector of an intersecting set have exactly one non-zero coefficient. The EKR-module property is an essential step in the characterization of the maximum intersecting sets.

It is obvious that the strict-EKR property implies the EKR-module property.
It is, however, possible for a transitive group to have the EKR-module property, but not the EKR property. Indeed this occurs if the largest intersecting set of permutations is the union of two or more canonical cocliques. An example is the group $\Alt(4)$ acting on unordered pairs from $\{1,2,3,4\}$. For this transitive group, the set of permutations mapping a pair $A$ to either $A$ or its complement $\overline{A}$ is an intersecting set of maximum size; this set is the union of two canonical cocliques.

We will define a graph with the property that an intersecting set in a group corresponds to a coclique in the graph. This graph has the property that it is connected if and only if the set of derangements generate the entire group. There are many examples of groups where this graph is not connected and, because of this, there can be many different maximum cocliques, and hence non-canonical maximum intersecting sets. In Section~\ref{sec:connected} we consider different cases when this graph is connected.

For 2-transitive groups with a connected derangement graph, all known examples
of non-canonical maximum intersecting sets  are, like canonical ones,
either subgroups or cosets of subgroups.
In Section~\ref{sec:subgroups}, we describe one way in which such non-canonical subgroups 
can arise for a $2$-transitive group with a regular normal subgroup.
These subgroups correspond to  elements of the first cohomology group of a point stabilizer with values in the regular normal subgroup. We consider two examples that illustrate what can happen in this situation.

The main result in this paper is that the characteristic vector of a maximum intersecting set in a 2-transitive group is a linear combination of the characteristic vectors for the canonical sets. In Section~\ref{sec:consequences} we prove that this result gives information about the structure of the set. Using an association scheme on the elements of the group, we can prove that any two maximum intersecting sets have the same {\it inner distribution}.
This is a count of the number of pairs $(g,h)$ of elements in the set that have  $hg^{-1}$ in a given conjugacy class. 

\section{Background}

In this paper we only consider 2-transitive permutation groups, so throughout
this paper $G$ is assumed to be a $2$-transitive group acting faithfully on
a set $X$ of size $n$. For each such
group we let $\chi_G$ denote the permutation character of this
2-transitive action. Since $G$ is 2-transitive, $\chi_G$ is the sum
of the trivial character (denoted $1_G$) and an irreducible character which we will
denote by $\psi_G$.

Let $\C[G]$ be the complex group algebra. The regular module can be identified with the vector space $\C[G]$ and given the structure of a  left $\C[G]$-module by left multiplication. Thus, $\C[G]$ also becomes identified with a subalgebra of the $\abs{G} \times \abs{G}$-matrices.
 
For any irreducible character $\phi$ of $G$, let $E_\phi$ to be the
$\abs{G} \times \abs{G}$-matrix with the $(g,h)$-entry equal to
$\frac{\phi(1)}{|G|} \phi(hg^{-1})$.  Then $E_\phi\in \C[G]$ is the primitive central idempotent
corresponding to $\phi$. We call the image of $E_\phi$ (considered as 
a linear operator on $\C[G]$) the $\phi$-module. It is an ideal of
$\C[G]$ of dimension $\phi(1)^2$. For the trivial representation the central idempotent is
$E_{1_G} = \frac{1}{ \abs{G} } J$, where $J$ is the all ones matrix.
We set $E_{\chi_{G}}=E_{1_G}+E_{\psi_G}$ and define the $\chi_G$-module
to be the image of $E_{\chi_{G}}$, an ideal of dimension $1+(n-1)^2$
in $\C[G]$. This leads to an equivalent definition of the EKR-module property,
from which its name originates.

\begin{lemma}\label{lem:module}
A 2-transitive group $G$ has the EKR-module property if and only if
the characteristic vector of any maximum intersecting set is in the
$\chi_G$-module. Equivalently, $G$ has the EKR-module property if and only if
$E_{\chi_G} v_S =v_S$ for any maximum intersecting set $S$. 
\end{lemma}
\begin{proof}
  This follows from two results from~\cite{MR3415003}. First, Lemma
  4.1 of~\cite{MR3415003} states that if $G$ is 2-transitive, then
  every $v_{i,j}$ is in the $\chi_G$-module. Lemma 4.2 of the same paper
  states that the vectors $v_{i,j}$ are a spanning set for the module.
\end{proof}

We also state a simple corollary of this lemma that gives the
result in a form that can be more convenient.

\begin{corollary}\label{lem:singlemodule}
If a 2-transitive group $G$ has the EKR-module property then for any maximum intersecting set $S$,
\begin{equation}\label{eq:proj}
E_{\psi_G} v_S =v_S - \frac{1}{n}\allone.
\end{equation} 
\end{corollary}
\begin{proof}
From Theorem~\ref{thm:2transEKR}, if $S$ is a maximum intersecting
set, then $E_{\bf 1}v_s = \frac{1}{n}\allone$ where
$\allone$ denotes the all-ones vector. Then Lemma~\ref{lem:module}
implies the equation. 
\end{proof}

A common approach to EKR theorems is to convert the problem to a
graph problem, and then apply techniques from algebraic graph theory (see~\cite{EKRbook} for details and examples). 
This is the approach that we will
use as well. The {\it derangement graph} of $G$ is the graph with
vertices the elements of $G$, in which two vertices are adjacent if
they are not intersecting.  The set of derangements 
(permutations with no fixed points)  in $G$ is denoted by $\Der_G$, and the
derangement graph of $G$ is denoted by $\Gamma_{G}$. The derangement
graph is the Cayley graph on $G$ with connection set $\Der_G$. A
coclique (or independent set) in $\Gamma_{G}$ is equivalent to a set
of intersecting permutations in $G$. Theorem~\ref{thm:2transEKR} can
be expressed as the size of a maximum coclique in $\Gamma_G$ is
$\frac{\abs{G}}{n}$ for any 2-transitive group $G$.

Using this graph structure allows us to use results from graph
theory. For example the clique-coclique bound (\cite[Corollary 2.1.2]{EKRbook})
easily translates to the following.

\begin{lemma}\label{lem:cliquecoclique}
Let $\omega(\Gamma_G)$ denote the size of the largest clique in
$\Gamma_G$, and $\alpha(\Gamma_G)$, the size of the largest
coclique. Then 
\[
\omega(\Gamma_G)\,\alpha(\Gamma_G) \leq |G|.
\]
Further, if equality holds, then each maximum clique intersects each
maximum coclique in exactly one vertex. \qed
\end{lemma}

We define a {\it normal Cayley graph} to be a Cayley graph with a connection set that is closed under conjugation. The graph $\Gamma_G$ is a normal Cayley graph since its connection set is the set of derangements in $G$.
The eigenvalues of a normal Cayley graph  can be calculated from the irreducible representations of $G$. The eigenvalue of $\Gamma_{G}$ belonging to the irreducible character $\phi$ of $G$ is
\[
\lambda_\phi = \frac{1}{\phi(1)} \sum_{d \in \Der_G} \phi(d).
\]
This result is usually attributed to Babai~\cite{Ba}, or Diaconis and
Shahshahani~\cite{MR626813}; a proof may be found in~\cite[Section
11.12]{EKRbook}.
The eigenvalue belonging to the trivial character is
clearly $d_G:=\abs{\Der_G} $, and it is not difficult to see that the
eigenvalue belonging to $\psi_G$ is
$-\frac{d_G}{n-1}$. Equation~\eqref{eq:proj} 
implies that if a
2-transitive group $G$ has the EKR-module property, then for any
maximum coclique $S$
\[
A(\Gamma_G) \left( v_S - \frac{1}{n}\allone \right) = -\frac{ d_G }{n-1} \left( v_S - \frac{1}{n}\allone \right)
\]
(where $A(\Gamma_G)$ is the adjacency matrix of $\Gamma_G$).

In his classic book Burnside showed \cite[\S134, Theorem IX]{Bu} that a 2-transitive group has a unique minimal normal subgroup. If this minimal normal subgroup is regular, then it is elementary abelian, and otherwise it is a non-abelian, primitive simple group
(see also \cite[Theorem~$4.1$B]{DM}).
We use this fact to divide the 2-transitive groups into two cases. In the next section we will prove Theorem~\ref{thm:main} for
2-transitive groups in which the minimal normal subgroup is
abelian. Section~\ref{sec:almostsimple} we will prove the result for
the groups in which the minimal normal subgroup is not
abelian; here we will need to use the classification of the almost simple $2$-transitive groups. We will consider when $\Gamma_G$ is connected in Section~\ref{sec:connected}.
Section~\ref{sec:subgroups} considers when the maximum
intersecting sets are groups or cosets of groups. In Section~\ref{sec:consequences} we
show that Theorem~\ref{thm:main} gives information
about the structure of the maximum intersecting sets. Finally we
discuss some questions for further investigation in Section~\ref{sec:conc}.

\section{2-transitive groups with a regular normal subgroup}
\label{sec:regNorm}

In this section we consider 2-transitive permutation groups $(G,X)$,
with $\abs{X}=n$, that have a regular normal subgroup $N$. In this
case, $N$ is an elementary abelian $p$-group for some prime
$p$. Further, $G$ is the semidirect product $N\stab$ where $\stab$ is
the stabilizer of a point $x \in X$. In particular, $\stab$ is a transversal of
$N$ in $G$ and $\stab$ is a coclique in $\Gamma_G$.

\begin{proposition}\label{prop:clique} 
The elements in $N$ form a clique of size $n$ in $\Gamma_G$.
\end{proposition}
\begin{proof}
Since $N$ is regular, it has size $n$ and every non-identity element
is a derangement. For any distinct $n_1, n_2 \in N$, $n_1n_2^{-1}$ is
a non-identity element of $N$, and is a derangement.
\end{proof}

By the clique-coclique bound (Lemma~\ref{lem:cliquecoclique}),
Proposition~\ref{prop:clique} implies that the size of a maximum
coclique in $\Gamma_G$ is bounded by $\frac{|G|}{n}$. Since $\stab$ is
a coclique of this size we have 
$\alpha(\Gamma_G) = \frac{\abs{G}}{n}$. This shows that all of these
groups have the EKR property. Further, any maximum coclique $S$ in
$\Gamma_G$ intersects $N$ (and any coset of $N$) in exactly one
element. So any coclique $S$ of maximum size is a transversal of $N$
in $G$. This can also be seen since for any two distinct elements $s$
and $t$ of $S$, the element $st^{-1}$ has a fixed point so does not
belong to $N$.

The following is a well-known result that we state in this context.

\begin{lemma}\label{lem:conj}
  Let $g=uh$ with $u\in N$ and $h\in \stab$. If $g$ is $G$-conjugate to an
  element of $\stab$, then the following hold:
  \begin{enumerate} [(a)] 
   \item $g=uh$ can be conjugated to $h$ by an element of $N$; and
   \item $h$ is the unique $N$-conjugate of $g$ in $\stab$.
  \end{enumerate}
\end{lemma}
\begin{proof}
  By hypothesis there exists $a\in G$ such that $a^{-1}ga\in \stab$. We
  may write $a=mk$, where $m\in N$ and $k\in \stab$. Then
  $k^{-1}m^{-1} (uh) mk\in \stab$, so $m^{-1}(uh) m \in k \stab k^{-1}=\stab$.

  As $\stab$ is a transversal of $N$ in $G$, two elements
  of $\stab$ with the same image in $G/N$ must be equal. Therefore the only possible
  $N$-conjugate of $uh$ in $\stab$ is $h$. So $m^{-1}uhm=h$ and both parts of the lemma are proved. 
\end{proof}

Let $S$ be a maximum coclique in $\Gamma_G$, for any elements $s, t
\in S$ (including $s=t$), write $st^{-1}=uh$ with $u\in N$ and
$h\in \stab$. As $uh$ has a fixed point, it is $G$-conjugate to an element
of $\stab$, hence $N$-conjugate to $h$ by Lemma~\ref{lem:conj}. If we fix $t$
and let $s$ run over $S$, then each element $h\in \stab$ is obtained in
this way exactly once, since $St^{-1}$ is also a transversal of $N$ in
$G$.  These observations will allow us, in the next lemma, to
generalize to arbitrary cocliques a calculation that was made for
canonical cocliques in~\cite[Lemma 4.1]{MR3415003}.

Recall that  $\psi_G$ denotes the irreducible character of $G$ of degree $n-1$
from the 2-transitive action.

\begin{lemma} Let $S$ be a coclique and $y\in G$.
\begin{equation}  \label{eq:coefficients}
  \sum_{s\in S} \psi_G(sy^{-1}) =
 \begin{cases}
    \abs{\stab} & \textrm{if } y\in S,\\
   \frac{-\abs{\stab}}{n-1} & \textrm{if } y\notin S.
\end{cases}
\end{equation}
\end{lemma}
\begin{proof}
  First suppose that $y\in S$. Write $sy^{-1}=uh$, where $u\in N$ and $h\in \stab$.
  We know from the previous lemma that $sy^{-1}$ is $G$-conjugate to $h$, and so $\psi_G(sy^{-1})=\psi_G(h)$.
  Moreover, as $s$ runs over $S$ we obtain each $h\in \stab$ once, so
 \begin{equation*}
    \sum_{s\in S} \psi_G(sy^{-1})=\sum_{h\in \stab} \psi_G(h) =\abs{\stab}.
\end{equation*}

Next suppose $y\notin S$. Since $S$ is a transversal of $N$ in $G$, we
can write $y=mt$, with $t\in S$ and $m$ a nonidentity element of
$N$. Suppose $st^{-1}=uh$, where $u\in N$ and $h\in \stab$. By
Lemma~\ref{lem:conj}, there exists $v\in N$ such that
$v(st^{-1})v^{-1}=v(uh)v^{-1}=h$.  Then
 \begin{equation*}
   \psi_G(sy^{-1})=\psi_G(st^{-1}m^{-1})=\psi_G(vst^{-1}m^{-1}v^{-1})
   =\psi_G(vst^{-1}v^{-1}m^{-1})=\psi_G(hm^{-1}).
 \end{equation*}
 Here we used the fact that $N$ is abelian.
 Moreover, the transversal property of $St^{-1}$
 means that, as $s$ runs over $S$, each element of $\stab$ is conjugate to
 $st^{-1}$ for exactly one $s$.
 Hence
 \begin{equation}\label{eq:cosetsum}
   \sum_{s\in S} \psi_G(sy^{-1})=\sum_{h\in \stab} \psi_G(hm^{-1}).
 \end{equation}
 Note that the right-hand side does not depend on $S$. This allows us to
 proceed as in the proof of ~\cite[Lemma 4.1]{MR3415003}.
 The right-hand side of (\eqref{eq:cosetsum})
 is the sum of $\psi_G$ over a coset of $\stab$ that is not equal $\stab$.
 By the $2$-transitivity of $G$ the value of this sum is the same  for all cosets of
 $\stab$ other than $\stab$ itself.
Then, since  $\sum_{g\in G}\psi_G(g)=0$ and $\sum_{g\in\stab}\psi_G(g)=\abs{\stab}$,
it follows that 
  \begin{equation*}
   \sum_{h\in \stab} \psi_G(hm^{-1}) =-\frac{\abs{\stab}}{n-1}.\qedhere
  \end{equation*}
\end{proof}

As in~\cite{MR3415003}, the sum computed in the
Equation~\eqref{eq:coefficients} is the coefficient of $y$ when the
element $\frac{\abs{G}}{\psi_G(1)}E_{\psi_G} v_S\in\C[G]$ is expressed in the
group basis. It follows as in \cite{MR3415003}, that
\begin{equation*}
  E_{\psi_G} \left( v_S-\frac1n\allone \right )=v_S-\frac1n\allone,
\end{equation*}
which shows that $v_S$ lies in the 2-sided ideal of $ \C[G] E_{\psi_G}$ of
$\C[G]$. This shows that $G$ has the EKR-module property, so Theorem~\ref{thm:main}
holds for any 2-transitive group with a regular normal subgroup.

\section{$2$-transitive groups of almost simple type}
\label{sec:almostsimple}

In this section we consider the 2-transitive groups that do not have a
regular abelian normal subgroup $N$; these are the $2$-transitive
groups of almost simple type. In this section, we assume that $G$ is
such a group and $K \trianglelefteq G$ is the minimal nonabelian
normal subgroup of $G$. These groups are listed in
Table~\ref{table0}. With the exception of $G =\Ree(3)$, for each of
these groups the subgroup $K$ is 2-transitive.  The eigenvalues of the
group $\Ree(3)$ can all be directly calculated, and $\psi_{\Ree(3)}$
is the only irreducible character affording the minimal
eigenvalue. Thus $\Ree(3)$ has the EKR-module property. So we will
restrict to the case where $K$ is 2-transitive.

We will show if $K$ has the EKR-module property, then $G$ also has the
EKR-module property. Then we will prove that each of these groups, the
minimal normal subgroup has the EKR-module property.

We assume that $G$ and $K$ are both acting on an $n$-set. We denote
character from this 2-transitive action of $G$ by $\chi_G$, and
$\chi_K$ is the representation of $K$ for its 2-transitive
action. Similarly, we use $\psi_G$ and $\psi_K$ for the irreducible
character of degree $n-1$ that is a component of $\chi_G$ and $\chi_K$.

\begin{lemma}\label{lem:lifting}
  Let $G$ be a 2-transitive group. If $S$ is a maximum coclique in $\Gamma_G$,
  then $v_s -\frac{1}{n}$ is a $\frac{-d_G}{n-1}$-eigenvector of
  $A(\Gamma_G)$.
\end{lemma}
\begin{proof}
From Theorem~\ref{thm:2transEKR}, the size of $S$ is $\frac{|G|}{n}$.

Since $S$ is a maximum coclique and $\Gamma_G$ is $d_G$-regular, the
number of edges between vertices in $S$ and vertices in
$V ( \Gamma_G )\backslash S$ is $d_G\abs{S}$.  So the quotient graph of $\Gamma_G$ with
the partition $\{ S, V(\Gamma_G)\backslash S\}$ is
\[
\begin{bmatrix}
0 & d_G \\
d_G \left( \frac{|S|}{|G|-|S|} \right) & d_G\left( 1 - \frac{|S|}{|G|-|S|}\right) \\
\end{bmatrix}.
\]
The eigenvalues of this quotient graph are $d_G$ and $-\frac{d_G}{n-1}$.
These eigenvalues interlace the eigenvalues of $\Gamma_G$. Further,
$d_G$ is the eigenvalue of $\Gamma_G$ afforded by the trivial
representation and $ -\frac{d_G}{n-1}$ is the eigenvalue afforded by
$\psi_G$. Since the eigenvalues of the quotient graph are eigenvalues
of the graph, the interlacing is tight. This means that
$\{S, G\backslash S\}$ is an equitable partition~\cite[Lemma
9.6.1]{MR1829620}.  So each vertex in $G\backslash S$ is adjacent to
exactly $d_G\frac{|S|}{|G|-|S|} $ vertices in $S$ and
$ d_G\left( 1- \frac{|S|}{|G|-|S|}\right)$ vertices not in $S$. By direct
calculation of $A(\Gamma_G)(v_S-\frac{1}{n})$,  the vector
$v_S-\frac{1}{n}$ is a $ -\frac{d_G}{n-1}$-eigenvector of $\Gamma_G$.
\end{proof}

\begin{lemma}\label{lem:Gder} 
  Suppose $H$ and $G$ are 2-transitive groups with $H\lneq G$. Then
  there exist derangements in $G$ that are not in $H$.
\end{lemma}
\begin{proof} 
We have
\begin{equation*}
    \sum_{g\in G}\chi_G(g)=\abs{G}\qquad\text{and}\qquad \sum_{h\in H}\chi_H(h)=\abs{H},
\end{equation*}
so
\begin{equation}\label{eq:notH}
    \sum_{x\in G\setminus H}\chi_G(x)=\abs{G\setminus H}.
\end{equation}

Suppose $\Der_G\subseteq H$. Then $\chi_G(x)\geq 1$ for all
$x\in G\setminus H$ so, by (\eqref{eq:notH}), we must have $\chi_G(x)=1$ and
$\psi_G(x)=0$ for all $x\in G\setminus H$.  

Since $G$ and $H$ both act
$2$-transitively, both $\psi_G$ and its restriction to $H$ are
irreducible characters. We have
  \begin{equation*}
    \sum_{g\in G}\psi_G(g)^2=\abs{G}\qquad\text{and}\qquad\sum_{h\in H}\psi_H(h)^2=\abs{H}.
  \end{equation*}
so
\begin{equation*}
    \sum_{x\in G\setminus H}\psi_G(x)^2=\abs{G\setminus H}.
\end{equation*}
Therefore, there exists $x\in G\setminus H$, with $\psi_G(x)\neq 0$. This contradiction completes the proof.
\end{proof}

\begin{theorem}
  Let $G$ be a 2-transitive group with minimal nonabelian normal
  subgroup $K$.  Assume $K$ is 2-transitive and that $\psi_K$ is the unique character of $K$ affording the
  least eigenvalue $-\frac{d_K}{n-1}$ of $\Gamma_K$. Then for any maximum
  coclique $S$ of $\Gamma_G$, $v_S-\frac{1}{n}\allone$ is in the $\psi_G$-module.
\end{theorem}
\begin{proof}
  Assume that $S$ is any maximum coclique of $\Gamma_G$.  Since $G$ is
  2-transitive, by Theorem~\ref{thm:2transEKR} $G$ has the EKR property, so the size of $S$ is
  $\frac{|G|}{n}$. By Lemma~\ref{lem:lifting}, $v_S-\frac{1}{n}\allone$ is
  a $-\frac{d_G}{n-1}$-eigenvector of $A(G)$.

Since $K$ is a subgroup of $G$, the graph $\Gamma_G$ contains $[G:K]$
copies of $\Gamma_K$ as a subgraph. Let $A$ be the adjacency matrix
for the $[G:K]$ copies of $\Gamma_K$. This is a weighted adjacency
matrix for $\Gamma_G$ where the edge $\{\sigma,\pi\}$ is weighted by
one if $\sigma\pi^{-1}$ is in the intersection of the derangements of
$G$ and $K$ (so $\sigma\pi^{-1}$ is a derangement in $K$), and zero
otherwise. 
The matrix $A$ has the form  $A = I_{[G:K]}\otimes A(\Gamma_K)$. Further, if
$G=\bigcup_{i=1}^{[G:K]} x_iK$ is the decomposition of $G$ into cosets of $K$,
then each $S_i = S \cap x_iK$ is a coclique of size
$\frac{|K|}{n}$ and each $v_{S_i} -\frac{1}{n}\allone$ is a
$-\frac{d_K}{n-1}$-eigenvector for $A$.  This means that
$v_{S} -\frac{1}{n}\allone$ is a $-\frac{d_K}{n-1}$-eigenvector for $A$.
The eigenvalues of $A$ are the same as the eigenvalues of $\Gamma_K$,
but the multiplicities of the eigenvalues for $A$ are equal to the
multiplicities of $\Gamma_K$ multiplied by $[G:K]$. In particular, the
eigenvalue $-\frac{d_K}{n-1}$ has multiplicity $[G:K](n-1)^2$ in $A$.

The induced character $\ind_G(1_K)$ has decomposition
\[
\ind_G(1_K) = \sum_{i} \phi_i(1) \phi_i, 
\]
where the $\phi_i$ are the distinct irreducible characters of $G$ having
$K$ in the kernel (which we may view as characters of $G/K$).
We choose notation so that $\phi_1=1_G$.
Then
\[
\ind_G(\psi_K) = \ind_G(1_K) \psi_G = \sum_{i} \phi_i(1) \phi_i \psi_G.
\]
Each  $\phi_i \psi_G$ is an irreducible character of
$G$ (\cite[Corollary 6.17]{MR2270898}). 
Since the restriction of $\ind_G(\psi_K)$ to $K$ equals
$[G:K]\psi_K$, the eigenvalue of $A$ afforded by each $\phi_i \psi_G$ is
$-\frac{d_K}{n-1}$. The dimension of the  sum of the $\phi_i\psi_G$-modules in $\C[G]$
equals $\sum_i(\phi_i(1)\psi_G(1))^2 =(n-1)^2\sum_i\phi_i(1)^2 =(n-1)^2[G:K]$,
so this sum is the entire $-\frac{d_K}{n-1}$-eigenspace of $A$.
Therefore, $v_{S} -\frac{1}{n}\allone$ lies the sum of the $\phi_i\psi_G$-modules.

Next we will use the fact that $v_{S} -\frac{1}{n}\allone$ is also a
$-\frac{d_G}{n-1}$-eigenvector for the adjacency matrix of $\Gamma_G$
to show that it is entirely contained in the $\phi_1 \psi_G$-module.

Consider
\[
\lambda_{\phi_i \psi_G} 
= \frac{1}{(n-1) \phi_i(1)} \sum_{d \in \Der_G} \phi_i(d) \psi_G(d)
= \frac{-1}{(n-1) \phi_i(1)} \sum_{d \in \Der_G}\phi_i(d). 
\]
By Lemma~\ref{lem:Gder} there are derangements in $G$ that are not
in $K$, so some $d$ we have $\phi_i(d) \neq\phi_i(1)$. So, if $\phi_i \neq 1_G$, then
\[
\frac{1}{\phi_i(1)} \sum_{d \in \Der_G}\phi_i(d) < \frac{1}{\phi_i(1)} \sum_{d \in \Der_G}\phi_i(1) =d_G.
\]
So no $\phi_i \psi_G$ affords $-\frac{d_G}{n-1}$ as an eigenvector, other
than $\phi_i = 1_G$. Since
$v_S -\frac{1}{n}\allone$ is both a $\frac{-d_G}{n-1}$ eigenvector and in
the sum of the $\phi_i \psi_G$-modules, it must in fact be in the $\psi_G$-module.
\end{proof}

The classification of finite simple groups has allowed for the
complete classification the finite $2$-transitive groups. Below is
Table 1 from~\cite{MR3474795} which lists the finite 2-transitive
groups of almost simple type. (This table was extracted from~\cite[page~197]{Ca}.)

\begin{table}[!h]
\begin{tabular}{ccccc}\hline
Line&Group $K$ &Degree&Condition on $G$&  Remarks \\\hline
1&$\Alt(n)$&$n$&$\Alt(n)\leq G\leq \Sym(n)$&$n\geq 5$ \\
2&$\PSL_m(q)$&$\frac{q^m-1}{q-1}$&$\PSL_m(q)\leq G\leq \mathrm{P}\Gamma\mathrm{L}_m(q)$&$m\geq 2$, $(m,q)\neq(2,2),(2,3)$\\
3&$\Sp_{2m}(2)$&$2^{m-1}(2^m-1)$&$G=K$&$m\geq 3$\\
4&$\Sp_{2m}(2)$&$2^{m-1}(2^m+1)$&$G=K$&$m\geq 3$\\
5&$\mathrm{PSU}_3(q)$&$q^3+1$&$\mathrm{PSU}_3(q)\leq G\leq \mathrm{P}\Gamma\mathrm{U}_3(q)$&$q\neq 2$\\
6&$\Sz(q)$&$q^2+1$&$\Sz(q)\leq G\leq \Aut(\Sz(q))$&$q=2^{2m+1}$, $m>0$\\
7&$\Ree(q)$&$q^3+1$&$\Ree(q)\leq G\leq \Aut(\Ree(q))$&$q=3^{2m+1}$, $m> 0$\\
8&$M_{n}$&$n$&$M_n\leq G\leq \Aut(M_n)$&$n\in \{11,12,22,23,24\}$,\\
&&&& $M_n$ Mathieu group,\\
&&&& $G=K$ or $n=22$\\
9&$M_{11}$&$12$&$G=K$&\\
10&$\PSL_2(11)$&$11$&$G=K$&\\
11&$\Alt(7)$&$15$&$G=K$&\\
12&$\PSL_2(8)$&$28$&$G=\mathrm{P}\Sigma\mathrm{L}_2(8) \cong \Ree(3)$&\\
13&$HS$&$176$&$G=K$&$HS$ Higman-Sims group\\
14&$Co_3$&$276$&$G=K$&$Co_3$ third Conway group\\\hline
\end{tabular}
\caption{Finite $2$-transitive groups of almost simple type}\label{table0}
\end{table}

\begin{proposition}\label{prop:Alt}
For $n\geq 5$ the least eigenvalue of $\Gamma_{\Alt(n)}$ is given by $\psi_{\Alt(n)}$ and no
other representations, and the largest eigenvalue is given by the
trivial character and no other. 
\end{proposition}
\begin{proof}
The number of derangements in $\Alt(n)$ is known~\cite[Sequence
A003221]{OEIS}, and for $n\geq 5$ we have 
\begin{align*}
d_{\Alt(n)} & = \frac{n!}{2} \sum_{i=0}^{n-2}  (-1)^i \frac{1}{i!} +
(-1)^{n-1}(n-1) \\
                   &\geq \frac{n!}{2} \left(1 - 1+\frac{1}{2} - \frac{1}{6} \right) = \frac{n!}{6}.
\end{align*}
The inequality clearly holds for $n$ odd. If $n$ is even and at least 6, then the inequality follows since 
$\frac{n!}{2(4!)} - \frac{n!}{2(5!)} - (n-1)$ is positive.

Using Lemma 2.4 from~\cite{MR3474795}, if the character $\phi \neq \psi_{\Alt}$ of the
alternating group affords the minimum eigenvalue of
$\Gamma_{\Alt(n)}$, then
\[
\phi(1) \leq (n-1) \left( \frac{|\Alt(n)|}{d_{\Alt(n)}} - 2 \right)^{\frac12}.
\]
Since
\[
(n-1)\left( \frac{|\Alt(n)|}{d_{\Alt(n)}} - 2 \right)^{\frac12}  \leq
(n-1)\left( \frac{n!}{2} \left(\frac{n!}{6}\right)^{-1} - 2 \right)^{\frac12} = n-1,
\]
any character giving the minimal eigenvalue must have dimension no
more than $n-1$. Since the only representations with degree no more
than $n-1$ are the trivial representation and $\psi_{\Alt}$, it follows 
that $\psi_{\Alt}$ is the unique irreducible representation affording
the minimum eigenvalue. Note that this also implies that only the
trivial representation gives the largest eigenvalue.
\end{proof}

\begin{theorem}
 The group $K$ of each type in Table~\ref{table0}  has $\chi_K$
  as the only irreducible character that gives the eigenvalue $-\frac{d_K}{n-1}$.
  (Here, as usual, we exclude $K=PSL(2,8)$ in row 12, as it is not $2$-transitive.)
\end{theorem}
\begin{proof}
  The previous result shows this holds for $\Alt(n)$ with $n \geq 5$. For $\PSL_2(q)$,
  this fact can be read off the tables in Simpson and
  Frame~\cite{MR0335618}, for $\PSL_3(q)$ it is in~\cite[Table 5]{KaPa2},
  and for $\PSL_m(q)$ with $m\geq 4$ it is stated in~\cite[Proposition
  8.3]{MR3474795}. For the groups in lines 3 and 4, $Sp_{2m}(2)$ this
  result is from~\cite[Proposition 9.1]{MR3474795} for $m\geq 7$. For
  $\PSU_3(q)$ this is from \cite[Table 5 and Table 6]{MR3923591}. For
  $Sz(q)$ the result is given in~\cite[Proposition 4.1]{MR3474795} and for $Ree(q)$
  this is \cite[Proposition 5.1]{MR3474795}. The eigenvalues of the
  Mathieu groups are given in \cite[Lemma 5.1]{MR3415003}. For all the
  other finite groups all the eigenvalues can be calculated from the
  character table, and only $\chi_K$ gives the eigenvalue
  $-\frac{d_K}{n-1}$.
\end{proof}

We have now shown that all 2-transitive groups of almost simple
type have the EKR-module property. In \S3 the EKR-module property was established
for 2-transitive groups with a regular normal subgroup. Thus, the proof of Theorem~\ref{thm:main} is complete.

\begin{remark}
  We have proven that the characteristic vector of every maximum coclique  in
  the derangement graph for a 2-transitive group is a linear combination of the characteristic vectors $v_{i, j}$ of the canonical cocliques.  We can say a little more  about the coefficients in this linear combination.   If $\Q[G] \subset \C[G]$ is the rational group algebra and $V$ is the intersection of $\Q[G]$ with the $\chi$-module, then $V$ is a $\Q[G]$-module whose dimension over $\Q$ equals the complex dimension of the  $\chi$-module. Since $E_\chi\in\Q[G]$, it follows that $E_\chi(\Q[G])= V$. The characteristic vector $v_S$ of any maximum coclique lies in $\Q[G]$ and  by the  EKR-module property we have $v_S=E_\chi(v_S)\in V$. Moreover,  the canonical  characteristic vectors $v_{i, j}$ span $V$ over $\Q$. Therefore, our proof actually shows that $v_S$ is a {\it rational} linear combination of the $v_{i, j}$.
\end{remark}

\section{Connected derangement graphs}
\label{sec:connected}

Consider the example of a 2-transitive Frobenius group $G$ with Frobenius kernel
$N$ and Frobenius complement $H$. (The group $\agl(1,q)$ where $q$ is a prime power is an example of such a group.) In this case, the cosets of $N$ are
cliques in the derangement graph of $G$. In fact, the derangement
graph is exactly the disjoint union of these cliques. Since any
transversal of $N$ is a coclique, as long as $\abs{H}>2$, there are
non-canonical cocliques of the form $H\setminus\{h\}\cup\{hu\}$, where
$h\in H$ and $u\in N$ are nonidentity elements.  The 2-transitive Frobenius groups with $\abs{H}>2$ are a family of groups that do not satisfy the strict-EKR
property. Further the non-canonical independent sets just described
are neither subgroups, nor cosets of subgroups. 

In this section we will consider other groups that have a
disconnected derangement graph; this occurs exactly when the derangements
do not generate the group.

\begin{lemma} Suppose $G$ contains a proper $2$-transitive subgroup $H$.
  Then $G$ is generated by $H\cup \Der_G$. In particular,
  if $H$ is generated by $\Der_H$, then $G$ is generated by $\Der_G$.
  \end{lemma}
\begin{proof} 
  Suppose for a contradiction that the subgroup $M$ of $G$ generated
  by $H\cup\Der_G$ is proper. Then we may apply Lemma~\ref{lem:Gder}
  to the group $G$ and the subgroup $M$, to obtain a derangement
  outside $M$. This is a contradiction and hence $M=G$. The last
  statement of the lemma follows immediately.
\end{proof}

For all the groups $G$ in Table~\ref{table0}, with the exception of
$\Ree(3)$, this corollary applies.  Proposition~\ref{prop:Alt} implies
that the derangement graph for the Alternating group is connected.
The fact that the minimal groups $K$ in lines 2-7 of Table~\ref{table0}
have a connected derangement graph can be read from~\cite{MR3474795}
(with results from~\cite{GT} for lines 3 and 4).  The groups in lines
8-11 and 13-14 are finite, and the eigenvalues of the derangement
graphs for the minimal group can be directly calculated and
individually checked.  With these facts, we have the following
corollary.

\begin{corollary}
  With the exception of $\Ree(3)$ (isomorphic to
  $\mathrm{P}\Sigma\mathrm{L}_2(8)$ with its action on 28 points), the
  derangement graph for any 2-transitive group of almost-simple type is
  connected.
\end{corollary}
\begin{proof}
  $\PSL_2(8)$ is a subgroup with index 3 in
  $\mathrm{P}\Sigma\mathrm{L}_2(8)$. Every element in
  $\mathrm{P}\Sigma\mathrm{L}_2(8)$ that is not in $\PSL_2(8)$ has
  order 3, 6 or 9 and $\psi_{\mathrm{P}\Sigma\mathrm{L}_2(8) }$
  vanishes on these points.  So all derangement of
  $\mathrm{P}\Sigma\mathrm{L}_2(8)$ are in $\PSL_2(8)$.
\end{proof}

Next we focus on the 2-transitive groups $G$ with a regular normal
subgroup $N$. We begin with an immediate consequence of
the fact that $\Der_G$ is a union of conjugacy classes.
\begin{lemma}
  Let $G$ be a $2$-transitive finite permutation group, with a regular
  normal subgroup $N$.  If $G/N\cong \stab$ is a simple group and
  there are derangements outside $N$, then the derangement graph of
  $G$ is connected.\qed
\end{lemma}

Fix an element $x$, from the set on which $G$ acts, and
let $H=\stab$ be its stabilizer. Then, by definition of the regular normal
subgroup, there is a map $N \rightarrow X$ defined by $u \mapsto u(x)$
that is an isomorphism of $N$ sets where $N$ acts on itself by left
multiplication. This is also an isomorphism of $H$-sets where $H$
acts on $N$ by conjugation. That is to say, for all $h\in H$ and
$u \in N$ we have $h(u(x)) = (huh^{-1})(x)$.

Under this identification of $N$ with $X$, the action of $G$ on $X$ is
equivalent to an action of $G$ on $N$ given as follows. Each element
of $G$ has the unique form $mh$ for $m \in N$ and $h \in H$. Then
$m h(u) = m(huh^{-1})$ for all $u\in N$.  We will make use of this $G$-action on $N$ in
the following lemmas.

\begin{lemma}\label{fh}
Let $G$ be a $2$-transitive finite permutation group
with a regular normal subgroup $N$ and point stabilizer $H$. Then
for $h\in H$, the coset $Nh$ contains a derangement if and only if $h$
centralizes a nonidentity element of $N$.
\end{lemma}
\begin{proof}
Consider the map $f_h:N \to N$ defined by 
\[
f_h(u)=huh^{-1}u^{-1}.
\] 
Then $h$ centralizes a nonidentity element of $N$ if and only if $f_h$
is not injective, which in turn is equivalent to $f_h$ not being
surjective. 

Suppose $f_h$ is not surjective, and let $m \in N$ be an element not
in the image of $f_h$. We claim that $m^{-1}h$ is a derangement. Here
we use the identification of $X$ with $N$ described above. Supposed
$m^{-1}h$ is not a derangement, then is has a fixed point. So
\begin{align}
u =(m^{-1}h)(u) = m^{-1}huh^{-1}
\end{align}
and it follows that $f_h(u) = m$, a contradiction. Thus if $f_h$ is
not surjective then $Nh$ contains a derangement.

Conversely, if $f_h$ is surjective, then for every $m\in N$, there
exists $u\in N$ such that $f_h(u) = m^{-1}$. This equation can be
written as $mhuh^{-1}=u$, that is $(mh)(u)=u$. Thus every element of
$Nh$ has a fixed point.
\end{proof}

\begin{theorem}
  Let $G$ be a $2$-transitive finite permutation group with a regular
  normal subgroup $N$ and point stabilizer $H=G_x$. Then the subgroup
  of $G$ generated by $\Der_G$ is equal to the subgroup generated by
  $N$ and the two-point stabilizer $H_y$, for $y\neq x$.
\end{theorem}
\begin{proof} Let $M$ be the subgroup of $G$ generated by $\Der_G$.
  Then $N\subseteq M$. By Lemma~\ref{fh}, a coset $Nh$, with $h\in H$
  contains a derangement if and only if $h$ centralizes a nonidentity element of $N$.
  In this case, the whole coset $Nh$ will be contained in $M$ since
  $N$ is contained in $M$. Thus, $M$ is equal to the subgroup generated by  those
  cosets $Nh$ for which $h$ centralizes a nonidentity element of $N$.

  As the conjugation action of $H$ on $N$ is isomorphic to the permutation action of $H$
  on $X$, an element $h$ centralizes a nonidentity element of $N$ if and only
  if $h$ lies in $H_y$ for some $y\in X$, $y\neq x$. This completes the proof.
\end{proof}
  
\begin{proposition}\label{prop:CharofFrobenius}
  Let $G$ be a $2$-transitive finite permutation group, with a regular
  normal subgroup $N$. Then $G$ is a Frobenius group if and only if
  $\Der_G=N\setminus\{1\}$.
\end{proposition}
\begin{proof} 

  If $G$ is a Frobenius group then it is immediate that
  $\Der_G=N\setminus\{1\}$. 

  Suppose that $G$ is not a Frobenius group. Then there is a
  nonidentity element $h\in H$ that centralizes a nonidentity element
  of $N$. Then by Lemma~\ref{fh}, the coset $Nh$ contains a
  derangement.
\end{proof}

\begin{corollary}
Let $G$ be a $2$-transitive finite permutation group, with a regular
  normal subgroup $N$. Then $G$ is a Frobenius group if and only if
  $\Gamma_G$ is the union of disjoint complete graphs.
\end{corollary}
\begin{proof}
  It is not hard to see that if $G$ is a Frobenius group, then
  $\Gamma_G$ is the union of complete graphs on $n$ vertices,
  see~\cite[Theorem 3.6]{MR3286446} for details.
  If $\Gamma_G$ is the union of disjoint complete graphs then, since a point stabilizer
  is a coclique of size $\abs{G}/n$, no complete subgraph has more than $n$ vertices.
  In particular, the identity element can have no more than $n-1$ neighbors.
  However the set of neighbors of the identity element is $\Der_G$, which
  contains $N \setminus\{1\}$, a set of size $n-1$. Thus,  $\Der_G= N \setminus\{1\}$, and by   Proposition~\ref{prop:CharofFrobenius} $G$ is a Frobenius group.
\end{proof}

There are many 2-transitive groups with a regular normal subgroup that
are not Frobenius groups and have disconnected derangement graphs. For
example, as we shall see, the groups  $\AGmL_1(p^e)$, for $p>2$ and $e\ge 2$,
are 2-transitive groups with a disconnected derangement graphs, 
and further examples may be found among their subgroups.
Each of these groups have the EKR-property, the EKR-module
property, but not the strict-EKR property. Further, for each of these
groups there are maximum cocliques that are neither subgroups, nor
cosets of subgroups.

\begin{proposition}
If $p>2$ is prime and $e\ge 2$ then $\AGmL_1(p^e)$ is a
2-transitive group with a disconnected derangement graph.
\end{proposition}
\begin{proof}
  Let $N$ be the regular normal subgroup of $\AGmL_1(p^e)$, consisting of
  the translations of the form $x\mapsto x+b$ with
  $b \in \mathbb{F}_{p^e}$). The two-point stabilizers of $\AGmL_1(p^e)$
  all have order $e$ and are generated by transformations  of the form
  $x\mapsto a^{(p-1)}x^p+b$ where $a,b \in \mathbb{F}_q$ and $a\neq 0$.
   These permutations do not generate all of $\AGmL_1(p^e)$.
\end{proof}

\section{Non-canonical Cocliques that are cosets of subgroups}
\label{sec:subgroups}

In this section we describe examples of groups that have a connected derangement graph, but also have noncanonical cocliques that are of maximum size.  These come from considering
non-canonical cocliques that are subgroups in $2$-transitive finite permutation groups $G$ with a regular normal subgroup $N$.

Since any coclique in $\Gamma_G$ must be a transversal of $N$, any
subgroup that is a non-canonical coclique must be complementary
to $N$, but not conjugate to $\stab$.  The $G$-conjugacy classes of
subgroups that are complementary to $N$, are classified by the first
cohomology group $H^1(\stab,N)$, where $N$ is viewed as an $\stab$-module by
conjugation. The trivial element of $H^1(\stab,N)$ corresponds to the
$G$-conjugacy class of $\stab$ and, if $H^1(\stab,N)$ is not trivial, each
nontrivial element corresponds to a $G$-conjugacy class of {\it
  nonstandard complements}, by which we mean subgroups complementary
to $N$, but not $G$-conjugate to $\stab$.

The following is a necessary and sufficient condition for a
nonstandard complement to be a maximum coclique in $\Gamma_G$.

\begin{lemma}\label{conj_condition} 
  A complement $K$ to $N$ in $G=N\stab$ is a coclique in $\Gamma_G$ if and
  only if every element of $K$ is $G$-conjugate to an element of
  $\stab$.
\end{lemma}
\begin{proof} 
  Assume that $K$ is a complement to $N$ that is a coclique in $\Gamma_G$. Since
  $1\in K$, each element of $K$ must have a fixed point (as it
  intersects with the identity element). Thus any element of $K$ lies in
  a point stabilizer and is $G$-conjugate to an element of
  $\stab$.

  Conversely, if each element of a subgroup $K$ is $G$-conjugate to an
  element of $\stab$, then every element has a fixed point. So for any
  $h,k \in K$, the element $hk^{-1} \in K$ has a fixed point which
  implies that $K$ is a coclique.
\end{proof}

Using the notation of the previous proof, let $g \in K$ and let $g_p$
be its $p$-part. It follows from the injectivity of the restriction of
$H^1(\langle g \rangle,N) \rightarrow H^1(\langle g_p\rangle,N)$
(see~\cite[Ch.XII, Theorem 10.1]{MR1731415}) that we may replace the condition
in Lemma~\ref{conj_condition} that every element of $K$ be
$G$-conjugate to an element of $H$, by the same condition on
$p$-elements only.

\begin{theorem}
 For $e\ge 2$, the group $\ASL_2(2^e)$ of affine transformations of
  $X=\mathbb{F}_{2^e}^2$ does not have the strict-EKR property.
\end{theorem}
\begin{proof}
Let $G=\ASL_2(2^e)$ be the group of affine
  transformations of $X=\mathbb{F}_{2^e}^2$ generated by the linear group
  $H=\SL_2(2^e)$ (the stabilizer of the zero vector) and the
  group $N=\mathbb{F}_{2^e}^2$ of translations, where $uh:x\mapsto hx+u$,
  for $x\in X$, $h\in H$ and $u\in N$.  It is well known that
  $H^1(H,N)\cong \mathbb{F}_{2^e}$ when $e\ge 2$~\cite[Lemma 14.7]{MR166272}.

  Since $H^1(H,N)$ is not trivial, this group has a non-standard complement, say $K$. The group $K$ is 
  not conjugate to $H$, but it is isomorphic to it. This implies that every element of $K$ is either an 
  involution or an element of odd order. Moreover, there is a single $K$-conjugacy class of involutions 
  and each involution in $K$ has the form $ut$, where $t\in H$ and $u\in C_N(t)$.

  If we regard $N$ as a $\mathbb{F}_{2^e}$-vector space and $t \in H$
  as a linear map, then $C_N(t)=\ker(t-1)$, and a simple calculation
  shows that $\ker(t-1)=\image(t-1)$.  It follows that for any $u\in N$
  there exists
  $m\in N$ such that $u=t^{-1}mtm^{-1}$, so $ut=tu=mtm^{-1}$ is
  conjugate to $t\in H$. As every odd order element of $K$ is
  conjugate to an element of $H$, the Schur-Zassenhaus Theorem,
  Lemma~\ref{conj_condition} shows that $K$ (and its cosets) are
  non-canonical cocliques in $\Gamma_G$.  
\end{proof}

\begin{example}
  For an explicit example, let $e=2$ and $\alpha$ be a primitive element
  of $\mathbb{F}_4$.  We can think of $\ASL_2(4)$ as the subgroup of
  $\SL_3(4)$ consisting of matrices of the block form
  \[
    \begin{bmatrix} A & v\\0&1   \end{bmatrix} 
  \]
where  $A\in \SL_2(4)$ and $v\in  \mathbb{F}_4^2$.

Consider the elements
  \begin{equation*}
  t=\begin{bmatrix} 1&0&0\\1&1&0\\0&0&1 
  \end{bmatrix}, \qquad
  u=\begin{bmatrix} 1&0&0\\0&1&1\\0&0&1 
  \end{bmatrix}
 \qquad  
  s=\begin{bmatrix} 0&1&0\\1&\alpha&0\\0&0&1 
  \end{bmatrix}    
  \end{equation*}
of orders $2$, $2$ and $5$ respectively.

The standard complement $H$ is generated by the elements $t$ and $s$,
while $tu$ and $s$ generate a nonstandard complement. It is
interesting to note that this non-standard complement has an orbit of size $6$
in $\mathbb{F}_{4}^2$ which is a {\it maximal arc of degree 2}.
(An arc of degree 2 is a subset in which no three points are collinear, and
in $\mathbb{F}_{4}^2$ such a subset can have at most $6$ points.)
\end{example}

Many other examples of non-canonical cocliques arising from nonstandard
complements can be found. However, it is not always the case that a
nonstandard complement will yield a non-canonical coclique in the
derangement graph, as it may fail to satisfy the hypotheses of
Lemma~\ref{conj_condition}, as in the following example.

\begin{example} 
Let $G=\AGL_3(2)=NH$, with $H=\GL_3(2)$ and
  $N=\mathbb{F}_{2}^3$, acting on $X=\mathbb{F}_{2}^3$ by affine transformations
  $uh:x\mapsto hx+u$, for $x\in X$, $h\in H$ and $u\in N$.  We can
  view $G$ as the subgroup of $GL_4(2)$ consisting of matrices of the
  following block form
\[
    \begin{bmatrix} A & v\\0&1  \end{bmatrix} 
\]
where  $A\in\GL_3(2)$ and $v\in \mathbb{F}_2^3$.

Consider the elements
  \begin{equation*}
  a=\begin{bmatrix} 1&1&0&0\\0&1&0&0\\0&0&1&0\\0&0&0&1 
  \end{bmatrix}, \qquad
  u=\begin{bmatrix} 1&0&0&1\\0&1&0&0\\0&0&1&1\\0&0&0&1 
  \end{bmatrix},
\qquad  
  s=\begin{bmatrix} 0&0&1&0\\1&0&1&0\\0&1&0&0\\0&0&0&1 
  \end{bmatrix}    
  \end{equation*}
of orders $2$, $2$ and $7$ respectively.

It is well known and easy to show by direct calculation that
$H^1(H,N)\cong\mathbb{F}_{2}$.
The standard complement $H$ is generated by the elements $a$ and $s$,
while $au$ and $s$ generate a nonstandard complement.  If $au$ were
$G$-conjugate to any element of $H$, it would be conjugate under $N$
to some element of $H$, as $G=NH$, and that element would have to have
the same image as $an$ in $G/N$. Thus $au$ would be conjugate to
$a$. However they are not conjugate in $G$, as $a-1$ and $au-1$ have
different ranks. So there are no subgroups that are also nonstandard cocliques.
This particular group can be shown directly to have the
strict-EKR property using the method described in~\cite{MR3415003}.
\end{example}

\section{Inner Distributions}
\label{sec:consequences}

We have proven that in any 2-transitive group  the characteristic vector of any maximum intersecting set is a linear combination of the characteristic vectors of the the canonical cocliques. For some groups, this fact has been used to show that the group has the strict-EKR property~\cite{GoMe,MR3780424, KaPa}. For other 2-transitive groups, that do not have the strict-EKR property, this fact has been used to characterize all the of the maximum intersecting sets~\cite{KaPa2, MR3921038}. 
It does not seem feasible to characterize the maximum intersecting sets for a general 2-transitive group, but 
in this section, we will prove that the EKR-module property does give us some extra information about the structure of the maximum intersecting sets. The number of pairs of elements $(g,h)$ in a set that have $hg^{-1}$ in a given conjugacy class is called the {\it inner distribution} of the set. We will show for a 2-transitive group, every maximum intersecting set has the same inner distribution. This gives information about the pair-wise intersection of elements within an intersecting set. In fact, this can been seen as a refinement of Lemma~\ref{conj_condition}. In Lemma~\ref{conj_condition} it was shown that if a complement is a coclique, then every element is conjugate to an element in $\stab$. The result on the inner distribution that we will prove in this section implies that
if a subgroup is a maximum coclique, then it has the same number of elements
in each conjugacy class of $G$ as $\stab$ has. 

To do this, we will consider the {\it conjugacy
  class scheme} on the group $G$. This is the association scheme that
has the elements of $G$ as its vertices and one class for each
conjugacy class of $G$. Two elements $g,h \in G$ are adjacent in a
class if $hg^{-1}$ is in the corresponding conjugacy class. The
matrices in this association scheme are indexed by the conjugacy
classes, and  denoted by $A_c$. The idempotents are indexed
by the irreducible representations of $G$, and denoted by
$E_\phi$.

Let $S$ be any maximum intersecting set in $G$. Let $v_S$ denote the
characteristic vector of $S$. Then the {\it inner distribution} of
$S$ is the sequence
\[
\left( \frac{v_S^T A_c v_S}{|S|} \right)_c
\]
taken over the conjugacy classes $c$ of $G$.
This gives a count of how many pairs of elements in $S$ are
$i$-related in the association scheme. The {\it dual distribution} is
defined to be the sequence 
\[
\left( \frac{v_S^T E_{\phi} v_S}{|S|} \right)_\phi
\]
taken over the irreducible representations $\phi$ of $G$.

Lemma~\ref{lem:module} implies for any maximum intersecting set
$S$ in $G$ that $v_S^T E_{\phi} v_S=0$, unless $\phi = 1_G$ or $\phi =
\psi_G$. From the comments following Lemma~\ref{lem:module}, we have
\[
\frac{v_S^T E_{1_G} v_S}{|S|} = \frac{|S|^2}{|G||S|} = \frac{1}{n}
\]
and
\[
\frac{v_S^T E_{\psi_G} v_S}{|S|} = 1 -\frac{1}{n}.
\]
Thus all maximal intersecting sets have the same dual distribution.
It is known (see~\cite[Theorem 3.5.1]{EKRbook}) that in any
association scheme the following equation holds
\[
\sum_{c} \frac{v_S^T A_c v_S}{|S|}A_c 
= \sum_{\phi} \frac{v_S^T  E_{\phi} v_S}{|S|} E_{\phi}.
\]
In particular, for any maximum intersecting set in $G$
\[
\sum_{c} \frac{v_S^T A_c v_S}{|S|}A_c 
= \frac{1}{n} E_{1_G} + \left( 1 -\frac{1}{n}\right) E_{\psi_G}.
\]
In the conjugacy class association scheme the sets $\{A_c \}$ and $\{ E_{\phi}\}$ are both
bases and the matrix of eigenvalues for the association scheme
is a change-of-basis matrix. The above equation implies that the inner distribution
for $S$ can be found by multiplying the dual distribution by the inverse
of the matrix of eigenvalues. In particular, we obtain the following result.

\begin{lemma}
Let $G$ be a 2-transitive group and let $S$ be any maximum
intersecting set in $G$. Then $S$ has the same inner distribution as
the stabilizer of a point. \qed
\end{lemma}

\section{Further Work}
\label{sec:conc}

There have been many papers looking at specific groups to determine
the structure of the maximum cocliques in the derangement
graph. Theorem~\ref{thm:main} gives a strong characterization of the
maximum cocliques in any 2-transitive groups. We end with an open
problem and a direction for further work.

Our only examples of groups that have non-canonical maximum cocliques
in their derangement graphs, that are neither subgroups nor cosets,
have the property that the derangement graphs are not connected. This
leads to our remaining question.

\begin{question}
  Are there 2-transitive groups $G$, with connected derangement graphs, that have a
  maximum coclique that is neither a subgroup nor a coset of a subgroup?
\end{question}

Finally, in this paper we only consider 2-transitive groups. The
definition of the EKR-module property can be considered for any group,
with the key difference being that, in general, the permutation module
is not the sum of the trivial module and a single irreducible
module. This situation will be more complicated, as there are
transitive groups which satisfy neither the EKR property, nor the
EKR-module property, nor the strict-EKR property. The first groups to
consider are the rank 3 groups.

\section*{Acknowledgment}
We thank the referee, whose numerous helpful suggestions have improved
the exposition of this paper.

\end{document}